\documentclass{article}
\usepackage{amsthm}
\usepackage{amssymb}
\usepackage{amsmath}
\usepackage{epsfig}
\usepackage{colonequals}
\usepackage{enumerate}
\usepackage{hyperref}
\usepackage[shortlabels]{enumitem}
\usepackage{tikz}

\newtheorem{theorem}{Theorem}

\newtheorem{corollary}[theorem]{Corollary}
\newtheorem{lemma}[theorem]{Lemma}
\newtheorem{remark}[theorem]{Remark}

\newtheorem{conjecture}[theorem]{Conjecture}

\newtheorem*{cor1}{Theorem \ref{thm:numc}}
\newtheorem*{cordeg}{Theorem \ref{thm-dense}}
\newtheorem*{lemcon}{Lemma \ref{lemma-concave}}
\newtheorem*{lemfun}{Lemma \ref{lemma-funnelweight}}
\newtheorem*{lemsla}{Lemma \ref{l:a}}
\newtheorem*{lemmain}{Lemma \ref{lemma-main-ineq}}

\DeclareMathOperator{\perm}{perm}

\begin{document}
\title{Bounding the number of cycles in a graph in terms of its degree sequence}
\author{ 
	Zden\v{e}k Dvo\v{r}\'ak\thanks{Computer Science Institute (CSI) of Charles University,
		Malostransk{\'e} n{\'a}m{\v e}st{\'\i} 25, 118 00 Prague, 
		Czech Republic. E-mail: \protect\href{mailto:rakdver@iuuk.mff.cuni.cz}{\protect\nolinkurl{rakdver@iuuk.mff.cuni.cz}}.
		Supported by the Neuron Foundation for Support of Science under Neuron Impuls programme.} \and Natasha Morrison\thanks{
	IMPA, Estrada Dona Castorina 110, Jardim Bot\^{a}nico, Rio de Janeiro, 22460-320, Brazil
	and Sidney Sussex College, Sidney Street, University of Cambridge, CB2 3HU, UK. E-mail: \protect\href{mailto:morrison@dpmms.cam.ac.uk}{\protect\nolinkurl{morrison@dpmms.cam.ac.uk}}.}
	\and Jonathan A. Noel\thanks{Mathematics Institute and DIMAP, University of Warwick, CV4 7AL Coventry, UK. E-mail: \protect\href{mailto: J.Noel@warwick.ac.uk}{\protect\nolinkurl{ J.Noel@warwick.ac.uk}}. Supported by the Leverhulme Trust Early Career Fellowship ECF-2018-534. } \and Sergey Norin\thanks{Department of Mathematics and Statistics, McGill University, Montr\'{e}al, Quebec, Canada.  E-mail: \protect\href{mailto: 
			sergey.norin@mcgill.ca}{\protect\nolinkurl{sergey.norin@mcgill.ca}}. Supported by an NSERC Discovery grant.} \and Luke Postle\thanks{Department of Combinatorics and Optimization,
		University of Waterloo,
		Waterloo, Ontario, Canada. E-mail: \protect\href{mailto: 
			lpostle@uwaterloo.ca}{\protect\nolinkurl{lpostle@uwaterloo.ca}}. Canada Research Chair in Graph Theory. Partially supported by NSERC
		under Discovery Grant No. 2019-04304, the Ontario Early Researcher Awards program and the Canada
		Research Chairs program.}}
\date{\today}
\maketitle
\begin{abstract}
We give an upper bound on the number of cycles in a simple graph in terms of its degree
sequence, and apply this bound to resolve several conjectures of Kir\'aly~\cite{kiraly2009maximum}
and Arman and Tsaturian~\cite{arman2017maximum} and to improve upper bounds on the maximum number of cycles in a planar graph.
\end{abstract}

\section{Introduction}

What is the largest number of cycles that can appear in a graph with $m$ edges?  Trivially, there
are at most $2^m$ cycles, and it is not hard to find examples with an exponential number of cycles. Kir\'aly~\cite{kiraly2009maximum} conjectured that the number of cycles in a graph with $m$ edges is $O(1.4^m)$. Our first main result proves this conjecture. 

\begin{theorem}\label{thm:numc}
Let $\gamma = 442368^{1/10} \approx 1.384$. A simple graph with $m$ edges has at most $O(\gamma^m)$ cycles.
\end{theorem}

Currently the best known lower bound is $\Omega(\kappa_1^m)$ for $\kappa_1=(2+2\sqrt{2})^{1/5}\approx 1.3701$ by Arman and Tsaturian~\cite{arman2017maximum} (see Figure~\ref{fig:k4chain}). 
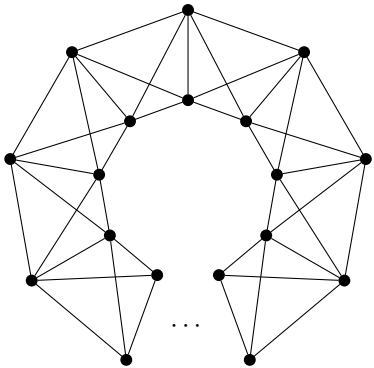
\begin{figure}
\begin{center}
\begin{tikzpicture}[scale=1.2,rotate=0]
\newdimen\R
\R=2cm
\newdimen\r
\r=1cm
\def\len{9}

\newcommand{\add}[2]{\dimexpr#1+#2\relax}

\newdimen\avg
\avg=\add{\R/2}{\r/2}

\tikzset{smallblack/.style={circle, draw=black!100,fill=black!100,thick, inner sep=0pt, minimum size=1.3mm}}

 \draw ({270 + (180/\len)}:\R)node[smallblack] (v1){}\foreach \x in {2,...,\len}
 {--({270+((\x-1)*360/\len) + (180/\len)}:\R) node[smallblack] (v\x){}};
 
 \draw(270:\avg)node(dots){$\cdots$};
 
 \draw ({270 + (180/\len)}:\r)node[smallblack] (u1){}\foreach \x in {2,...,\len}
 {--({270+((\x-1)*360/\len) + (180/\len)}:\r) node[smallblack] (u\x){}};
 
 \foreach \x in {1,...,\len}
 \draw(u\x)--(v\x);
 
 \foreach \x [count=\xx from 1] in {2,...,\len}
 \draw(u\x)--(v\xx);
 \foreach \x [count=\xx from 1] in {2,...,\len}
 \draw(v\x)--(u\xx);
\end{tikzpicture}
\caption{The lower bound construction of Arman and Tsaturian~\cite{arman2017maximum} giving simple graphs with $\Omega(\kappa_1^m)$ cycles for $\kappa_1=(2+2\sqrt{2})^{1/5}\approx 1.3701$.}\label{fig:k4chain}
\end{center}
\end{figure}

We believe that this construction is the best possible
(we verified by computer a number of similar constructions and could not find a better one).
\begin{conjecture}\label{conj-cyc}
Every simple graph with $m$ edges has at most $O(\kappa_1^m)$ cycles.
\end{conjecture}

In fact, Theorem~\ref{thm:numc} is a consequence of a much more general result (Theorem~\ref{thm-paths}). For a graph $G$ with (not necessarily distinct) vertices $s$ and $t$, let $p_{s,t}(G)$ denote
the number of distinct paths from $s$ to $t$ in $G$. Theorem~\ref{thm-paths} provides a bound on $p_{s,t}(G)$ as a function of the degree sequence of $G$, which easily implies a bound on the number of cycles in $G$ (see the proof of Theorem~\ref{thm:numc} for details). As Theorem~\ref{thm-paths} is fairly technical, we defer the statement to Section~\ref{sec:main_results}.

Theorem~\ref{thm-paths} also yields new bounds for the number of cycles in simple graphs with small average degree. For average degree at most four, the following theorem improves on a result of Kir\'aly~\cite{kiraly2009maximum} of $O((\sqrt{2}-\varepsilon)^m)$
for a very small (but fixed) $\varepsilon>0$. It also proves his conjecture that 4-regular graphs have at most $1.99^{m/2}$ cycles.

\begin{theorem}\label{thm:4-reg}
A simple graph  with $m$ edges and average degree 4 has $O(1.365^m)$ cycles. 
\end{theorem}
Patk\'os~\cite{patkos} gave a construction for $4$-regular graphs with at least $\kappa_2^m$ cycles, where $\kappa_2\approx 1.356$.

Bounds on the maximum number of cycles in planar graphs on $n$ vertices have been previously given by Alt, Fuchs and Kriegel~\cite{AFK99} and  improved upon by Buchin et al.~\cite{BKKSS07}, who have proven a lower bound of $\Omega(2.4262^n)$ and an upper bound of $O(2.8927^n)$. We are also able to improve this bound. 

\begin{theorem}\label{thm:planar}
A planar graph with $n$ vertices has $O(2.643^n)$ cycles.
\end{theorem}

Both Theorem~\ref{thm:4-reg} and Theorem~\ref{thm:planar} follow from a more general (but more technical) result (Corollary~\ref{cor-avgdeg}) which gives upper bounds for the number of cycles in simple graphs with constant average degree.
The same approach implies the following asymptotic bound, of interest for graphs with superconstant average degree. 

\begin{theorem}\label{thm-dense}
A simple graph with $n$ vertices and average degree $d$ has at most
$$\left(\left(1+O\Bigl(\frac{\log d}{d}\Bigl)\right)\cdot \frac{d}{e}\right)^n$$
cycles.
\end{theorem}

This proves Conjecture~6.1 of Arman and Tsaturian~\cite{arman2017maximum}, and is asymptotically tight as seen
by considering random graphs~\cite{glebkriv}.  In combination with the Erd\H{o}s--Stone theorem~\cite{erdos1946structure}, Theorem~\ref{thm-dense}
implies Theorem~1.4 of Morrison, Roberts and Scott~\cite{morrison2019maximising} that for a fixed graph $H$ of chromatic number
$\chi\ge 3$, a simple $n$-vertex graph not containing $H$ as a subgraph contains at most $\Bigl((1+o(1))\frac{(\chi-2)n}{(\chi-1)e}\Bigr)^n$
cycles. There is also a short direct proof of Theorem~\ref{thm-dense} using Br\`{e}gman's Theorem~\cite{Bregman} (see also~\cite{Schrijver78}) on the permanent of a matrix, which we present in Section~\ref{sec:main_results}.

The key idea in the proof of Theorem~\ref{thm-paths} is to prove a bound in terms of the degree sequence of the graph, which enables us to use inductive arguments. In Section~\ref{sec:main_results}, we introduce this idea by giving a short proof
of an analogous result for multigraphs (Theorem~\ref{thm-multi}). We will then formally state Theorem~\ref{thm-paths} and deduce Theorem~\ref{thm:numc}, Corollary~\ref{cor-avgdeg} (which will imply Theorems~\ref{thm:4-reg} and \ref{thm:planar}) and Theorem~\ref{thm-dense} from it. The proof of Theorem~\ref{thm-paths} is then given in Section~\ref{sec:main_proof}. Unfortunately, this proof relies on a large number of technical calculations. In order to not obscure the key ideas of the proof from the reader, some of these are deferred to Section~\ref{sec:post_proof}.

\section{Main Results}\label{sec:main_results}

\subsection{A  bound for multigraphs}
To introduce the main ideas in the proof of Theorem~\ref{thm-paths}, let us give a short proof of an analogous bound for multigraphs (Theorem~\ref{thm-multi} below). Our bound is a strengthening of a result of Arman and Tsaturian~\cite{arman2017maximum}, who proved that the number of cycles in a multigraph with $m$
edges is at most $O(3^{m/3})$. This bound is tight, as illustrated by the graph
obtained from the cycle of length $m/3$ by replacing each edge by a triple edge.

In order to state Theorem~\ref{thm-multi} we will first introduce some key definitions. 
We will define functions $a', b'$ and $q'$; analogous functions $a, b$ and $q$ will be defined later for the proof of Theorem~\ref{thm-paths}. Let $b'(0)=a'(0)=b'(1)=1$.  Let $b'(k)=k/2$ for $k \geq 2$, and let $a'(k)=2\sqrt{k}$ for $k\ge 1$.  For a graph $G$ with a vertex $s$, let
$q'_s(G)=a'(\deg(s))\prod_{v\in V(G)\setminus \{s\}} b'(\deg(v))$. Recall that for a graph $G$ with (not necessarily distinct) vertices $s$ and $t$, we define $p_{s,t}(G)$ to be
the number of distinct paths from $s$ to $t$ in $G$.

\begin{theorem}\label{thm-multi}
If $G$ is a multigraph, then $p_{s,t}(G)\le q'_s(G)$ for all $s,t \in V(G)$.
\end{theorem}
\begin{proof}
We prove the claim by induction on the number of vertices of $G$.  
As $a'(k)$ and $b'(k)$ are non-decreasing, and $b'(k) \geq 1$ for every $k$, we have $q'_s(G') \leq q'_s(G)$ for every subgraph $G'$ of $G$ such that $s \in V(G')$. Thus we can assume $G$ is connected,
$s\neq t$, and every cutvertex of $G$ is distinct from $s$ and $t$ and separates $s$ from $t$. 
In particular, this implies all vertices distinct from $s$ and $t$ have degree at least two.
  
Suppose that a vertex $v \neq s$ has degree zero in $G - s$. It follows from the previous paragraph that $v=t$ and $V(G)=\{s,t\}$. In this case, letting $l$ be the
multiplicity of the edge between $s$ and $t$, we have $q'_s(G) \ge l^{3/2} \geq l = p_{s,t}(G)$.

Suppose next that a vertex $v \neq s$ has degree one in $G - s$, and  let $w$ 
be the unique neighbor of $v$ in $V(G) \setminus \{s\}$.   Let $G'$ be obtained from $G$ by contracting the edge $vw$ (retaining any multiple edges which arise).
Then $p_{s,t}(G)=p_{s,t}(G')$; note this is the case even if $v=t$ (here, we let $t$ denote the vertex obtained by the contraction).
Observe that $$b'(\deg(v))b'(\deg(w))\ge b'(\deg(v)+\deg(w)-2),$$
and thus we have $q'_s(G) \geq q'_s(G')$. Hence, $p_{s,t}(G)\le q'_s(G)$ follows by applying the induction hypothesis to $G'$.

It remains to consider the case when every vertex $v \neq s$ has degree at least two in $G - s$. Let $r_1$, \ldots, $r_c$ be the neighbors of $s$ in $G$, for $i=1,\ldots,c$ let
$l_i$ denote the multiplicity of the edge $sr_i,$ and let $n_i=\deg(r_i)-l_i$; we have $n_i\ge 2$.
Let $k=\deg(s)=\sum_{i=1}^c l_i$ and $\beta=\prod_{v\in V(G)\setminus\{s,r_1,\ldots, r_c\}} b'(\deg(v))$.
Applying the induction hypothesis to $G-s$, we obtain
\begin{align*}
p_{s,t}(G)&=\sum_{i=1}^c l_ip_{r_i,t}(G-s)\le \sum_{i=1}^cl_iq'_{r_i}(G-s)\\
&=\beta\cdot\Bigl(\sum_{i=1}^c \frac{l_ia'(n_i)}{b'(n_i)}\Bigr)\cdot\prod_{i=1}^c b'(n_i).
\end{align*}
Since $q'_s(G)=a'(k)\beta\cdot \prod_{i=1}^c b'(n_i+l_i)$, it suffices to prove that
$$a'(k)\prod_{i=1}^{c}\frac{b'(n_i+l_i)}{b'(n_i)}\geq \sum_{i=1}^{c}\frac{l_ia'(n_i)}{b'(n_i)}.$$
That is, we need to prove
\begin{equation}\label{eq-main}
\sqrt{\sum_{i=1}^c l_i}\cdot \prod_{i=1}^{c}\Bigl(1+\frac{l_i}{n_i}\Bigr)\geq 2\sum_{i=1}^{c}\frac{l_i}{\sqrt{n_i}}.
\end{equation}
By the Cauchy-Schwarz inequality applied to vectors $(\sqrt{l_i}:i=1,\ldots,c)$ and $(\sqrt{l_i/n_i}:i=1,\ldots,c)$, we have
$$\sum_{i=1}^{c}\frac{l_i}{\sqrt{n_i}}\le \sqrt{\sum_{i=1}^c l_i}\cdot\sqrt{\sum_{i=1}^c \frac{l_i}{n_i}},$$
and thus it suffices to prove that
$$\prod_{i=1}^{c}\Bigl(1+\frac{l_i}{n_i}\Bigr)^2\ge 4\sum_{i=1}^c \frac{l_i}{n_i}.$$
This is the case, since
$$\prod_{i=1}^{c}(1+x_i)^2\ge 4\sum_{i=1}^c x_i$$
holds for all real numbers $x_1,\ldots,x_c\ge 0$.
\end{proof}
Note that $b'(d)\le 3^{d/6}$ for every $d\ge 0$ (with equality for $d=6$). Therefore for a multigraph $G$
with $m$ edges, we have
$$\prod_{v\in V(G)} b'(\deg(v))\le 3^{\frac{1}{6}\sum_{v\in V(G)} \deg(v)}=3^{m/3},$$
and thus $p_{s,t}(G)=O(3^{m/3})$.

\subsection{Introducing Theorem~\ref{thm-paths}}

Our main result is a variation on Theorem~\ref{thm-multi} which gives better bounds, but only applies to simple graphs.
For the rest of the paper, all graphs are simple, i.e.  without loops or parallel edges, and all logarithms are natural

As in the previous section we begin by defining functions that will be used in the proof (we will motivate these definitions after stating Theorem~\ref{thm-paths}). 
Let $b(0)=b(1)=1$.
Let $\rho = (4/3)^{1/4}$, and for every integer $k\ge 2$, let $b(k)=((k-1)!\rho)^{1/(k-1)}$.
Hence, $b(2)=\rho$ and we have
\begin{equation}\label{e:brecursion}
b(k+1)^k = kb(k)^{k-1}
\end{equation}
for every $k \geq 2$.  Note that $b(k)<b'(k)$ for every $k\ge 3$
(we have $b(k)=(k+\tfrac{1}{2}\log k)/e+O(1)$ by Stirling's bound), which is the main source of the improvement over Theorem~\ref{thm-multi}.
Let
$$\lambda(k)=\log \tfrac{b(k+1)}{b(k)}=\frac{1}{k(k-1)}\log\frac{k^k}{\rho k!}$$
for $k\ge 2$.
Let $\alpha(0)=1$, $\alpha(1)=\sqrt{\tfrac{2}{\rho^5}}$, 
$\alpha(2)=\tfrac{2}{\rho^3}$, $\alpha(3)=\tfrac{6}{\rho^4 b(3)b(5)}$ and let 
$$\alpha(k)=\frac{\alpha(k-1)}{1+k(\lambda(k-1)-\lambda(k))}$$ for $k \geq 4$. Finally, let $a(k)=\alpha(k)b(k)$ for every integer $k\ge 0$.
For a graph $G$ and a vertex $s$ of $G$,
we define $$q_s(G)=a(\deg(s))\prod_{v\in V(G)\setminus\{s\}} b(\deg(v)).$$
Given these definitions, we are able to state Theorem~\ref{thm-paths}.
\begin{theorem}\label{thm-paths}
If $G$ is a simple graph, then $p_{s,t}(G)\le q_s(G)$ for all $s,t \in V(G)$.
\end{theorem}

Let now us briefly motivate the values of $b$.
Consider the graph $G$ obtained from a path with ends $s$ and $t$ and with $m$ edges by, for each edge $e$,
adding a vertex adjacent to both ends of $e$.  Then $G$ contains $2^m$ paths from $s$ to $t$, $m+1$ vertices of degree two
distinct from $s$, and $m-1$ vertices of degree $4$.  Hence, we need $(a(2)b(2)/b(4))\cdot (b(2)b(4))^m\ge 2^m$
to hold for every $m$, and thus we need $b(2)b(4)\ge 2$.
Furthermore, for the inductive argument, we need the inequality (\ref{e:main2}),
analogous to the inequality (\ref{eq-main}) from Theorem~\ref{thm-multi}, to hold.
When $c=k$, $l_1=\ldots=l_k=1$ and $n_1=\ldots=n_k=k$ for an integer $k\ge 2$,
this gives $b(k+1)^k\ge kb(k)^{k-1}$. The choice of values of $b$ ensures that these inequalities are tight.

As for the choice of the function $a$, the value of $a(2)$ comes from considering a triangle with distinct vertices $s$ and $t$.
Other values for $a$ again follow from the inequality (\ref{e:main2}).
The value of $a(3)$ corresponds to the case $c=1$, $k=3$, $l_1=3$ and $n_1=2$.
For larger $k$, the value $a(k)$ can actually be chosen by considering the case that
$c=k$, $l_1=\ldots=l_k=1$, $n_1=\ldots=n_{k-1}=k$ and $n_k=k-1$.
However, for the purposes of the proof, it is more convenient to use our choice, obtained
by extrapolating the formulas to real numbers and considering the limit case with $(k-\varepsilon)$ of the values of $n_i$
equal to $k$ and $\varepsilon$ of the values equal to $k-1$, for $\varepsilon\to 0$.

\subsection{Implications of Theorem~\ref{thm-paths}}

In this subsection we will deduce Theorems~\ref{thm:numc},  \ref{thm:4-reg}, \ref{thm:planar} and \ref{thm-dense} from Theorem~\ref{thm-paths}.
As mentioned in the introduction, a bound on $p_{s,t}$ easily translates to a bound on the number of cycles; we will now make this rigorous in our proof of Theorem~\ref{thm:numc}. Let $\gamma=\sup\{b(k)^{2/k}:k\in\mathbb{N}\}=b(5)^{2/5}\approx 1.38403$.
\begin{cor1}[Restated]
A simple graph with $m$ edges has at most $O(\gamma^m)$ cycles.
\end{cor1}
\begin{proof}
Let $G$ be a graph with $m$ edges. Let $\ell=\lfloor m/11\rfloor$.
Note that $G$ has at most
$$\sum_{j=3}^{\ell} \binom{m}{j}  \leq \ell\binom{m}{\ell}\le m\Bigl((em/\ell)^{\ell/m}\Bigr)^m<m1.37^m=o(\gamma^m)$$
cycles of length at most $\ell$.
Let $\mathcal{C}$ be the set of cycles of $G$ of length greater than $\ell$.
For an edge $e\in E(G)$ with ends $x$ and $y$, let $G_e$ be the
graph obtained from $G-e$ by adding vertices $s$ and $t$ adjacent to $x$ and $y$, respectively.
Observe that
$$\sum_{e\in E(G)} p_{s,t}(G_e) \ge \sum_{C\in \mathcal{C}} |E(C)|\ge m|\mathcal{C}|/11.$$
Let $\beta=\prod_{v\in V(G)} b(\deg(v))$.
By Theorem~\ref{thm-paths}, we have $p_{s,t}(G_e)\le q_s(G_e)=\alpha(1)\beta<2\beta$ for every edge $e$.
Hence, $G$ contains less than $22\beta$ cycles of length greater than $\ell$.  Furthermore,
$$\beta=\prod_{v\in V(G)} \Bigl(b(\deg(v))^{2/\deg(v)}\Bigr)^{\deg(v)/2}\le \gamma^{\frac{1}{2}\sum_{v\in V(G)} \deg(v)}=\gamma^m,$$
giving the stated bound.
\end{proof}

More generally, since $\log b(k)$ is concave for $k\ge 2$ (see Lemma~\ref{lemma-concave}),
among graphs of average degree $d$, the product $\prod_{v\in V(G)} b(\deg(v))$ is maximized by
graphs with only vertices of degreees $\lfloor d\rfloor$ and $\lceil d\rceil$. Hence, the same
argument gives the following bound depending on the average degree.
For a real number $d\ge 2$, let $\gamma_d=b(d)^{2/d}$ if $d$ is an integer and
$$\gamma_d=b(\lfloor d\rfloor)^{2(\lceil d\rceil-d)/d}b(\lceil d\rceil)^{2(d-\lfloor d\rfloor)/d}$$ otherwise. Thus we obtain the following corollary of Theorem~\ref{thm-paths}.

\begin{corollary}\label{cor-avgdeg}
A simple graph with $m$ edges, minimum degree $2$ and of average degree $d\ge 2$ has at most $m\gamma_d^m$ cycles.
\end{corollary}

For graphs with constant average degree, an argument similar to that in the proof of Theorem~\ref{thm:numc} gives a bound of $O(\gamma_d^m)$. As $\gamma_4 < 1.365$, this immediately implies Theorem~\ref{thm:4-reg}. As the  average degree of every planar graph is less than six and $mb(6)^m=O(2.643^m)$, we also get Theorem~\ref{thm:planar}. 

\begin{remark}
Observe that $\gamma_d>\kappa_1$ only when $4.24<d<7.18$; hence, by Corollary~\ref{cor-avgdeg}, counterexamples to Conjecture~\ref{conj-cyc}
would have to have average degree in this range.
\end{remark}

Analogously to Corollary~\ref{cor-avgdeg}, as $b(k)=((k-1)!\rho)^{1/(k-1)}=\frac{k+O(\log k)}{e}$,
Theorem~\ref{thm-paths} immediately gives Theorem~\ref{thm-dense}. 

\begin{cordeg}[Restated]
A simple graph with $n$ vertices and average degree $d$ has at most
$$\left(\left(1+O\Bigl(\frac{\log d}{d}\Bigl)\right)\cdot \frac{d}{e}\right)^n$$
cycles.
\end{cordeg}

For the interested reader, we now present a short proof of Theorem~\ref{thm-dense} using Br\`{e}gman's Theorem~\cite{Bregman}. Recall that the permanent of $A = (a_{i,j})_{1\leq i,j\leq n}$ is 
\[\perm(A) = \sum_{\sigma\in S_n}\prod_{i=1}^na_{i,\sigma(i)}.\]

\begin{theorem}[Br\`{e}gman's Theorem~\cite{Bregman}]\label{thm:breg}
Let $A$ be an $n \times n$ matrix with all entries in $\{0,1\}$ and row sums $(r_i)_{1 \le i \le n}$. Then $\perm(A) \le \prod_{i=1}^n (r_i!)^{1/r_i}.$
\end{theorem}

\begin{proof}[Proof of Theorem~\ref{thm-dense}]
Let $A$ be the adjacency matrix of $G$ and define $A'=A+I$ where $I$ is the $n\times n$ identity matrix. In particular, $\perm(A')$ is bounded below by the number of cycles in $G$. Indeed, for each cycle $C$, consider the permutation $\sigma$ which maps each vertex of $C$ to the vertex which comes immediately after it on $C$ and maps every vertex of $V(G)\setminus V(C)$ to itself. By Br\`{e}gman's Theorem (Theorem~\ref{thm:breg}), convexity (using, e.g.,~Karamata's inequality~\cite{Karamata}) and Stirling's bound (see Lemma~\ref{lemma-stirling}), the number of cycles in $G$ is therefore at most
\[\prod_{v\in V(G)}\left(d(v)+1\right)!^{1/(d(v)+1)}\leq (d+1)!^{n/(d+1)}
\leq\left(\left(1+O\Bigl(\frac{\log d}{d}\Bigl)\right)\cdot \frac{d}{e}\right)^n.\]
\end{proof}

\section{Proof of Theorem~\ref{thm-paths}}\label{sec:main_proof}

In order to prove Theorem~\ref{thm-paths}, we use the following properties of the functions $a$ and $b$.
Since their proofs are rather technical (especially in the case of Lemma~\ref{lemma-main-ineq}), we postpone them for now.

\newcommand{\sconcave}{For $k\ge 2$, the sequence $b(k)$ is increasing, and the sequence $\tfrac{b(k+1)}{b(k)}$ is decreasing.}
\begin{lemma}\label{lemma-concave}
\sconcave
\end{lemma}

\newcommand{\sla}{For $k\ge 0$, the sequence $a(k)$ is increasing.}
\begin{lemma}\label{l:a}
\sla
\end{lemma}

For integers $l\ge d\ge 1$, let $\phi(l,d)=\rho^{l-2}b(l-d+2)$.

\newcommand{\sfunnelweight}{For all integers $l\ge d\ge 1$, we have $a(l)\phi(l,d)b(d)\ge l$.}
\begin{lemma}\label{lemma-funnelweight}
\sfunnelweight
\end{lemma}

\newcommand{\smain}{Let $c$, $l_1, \ldots, l_c$, $d_1, \dots, d_c$ and $n_1,\ldots,n_c$ be integers such that
$l_i\ge d_i\ge 1$ and $n_i\ge 2$ for $1\le i\le c$.
Then
$$a\Bigl(\sum_{i=1}^c l_i\Bigr)\prod_{i=1}^c \phi(l_i,d_i)b(n_i+d_i)\ge \Bigl(\sum_{i=1}^c l_i\alpha(n_i)\Bigr)\cdot\prod_{i=1}^c b(n_i).$$}
\begin{lemma}\label{lemma-main-ineq}
\smain
\end{lemma}

A triple $(G,s,t)$ is a \emph{counterexample} to Theorem~\ref{thm-paths} if $G$ is a graph with specified vertices $s$ and $t$
satisfying the assumptions of the theorem but containing more than $q_s(G)$ paths from $s$ to $t$.
Since $a(k),b(k)\ge 1$ for every $k$ and $p_{s,s}(G)=1$, a counterexample satisfies $s\neq t$.
A counterexample $(G,s,t)$ is \emph{smallest} if $G$ has the smallest number of vertices among all counterexamples.

For a set $Z\subseteq V(G)$, we say $G$ is \emph{$(Z,2)$-connected} if $G$ is connected
and for every $v\in V(G)$, every component of $G-v$ intersects $Z$.

\begin{lemma}\label{lemma-2conn}
Every smallest counterexample $(G,s,t)$ to Theorem~\ref{thm-paths} is $(\{s,t\},2)$-connected.
\end{lemma}
\begin{proof}
Clearly, $G$ is connected.  Suppose for a contradiction that there exists a vertex $v$ and a non-empty
set $K\subseteq V(G)\setminus\{w,s,t\}$ such that no vertex of $K$ has a neighbor outside of $K\cup\{v\}$. Since both $a$ and $b$ are non-decreasing by Lemmas~\ref{lemma-concave} and \ref{l:a},
and since $b(k)\ge 1$ for every $k\ge 0$, we have
$q_s(G-K)\le q_s(G)$, while $p_{s,t}(G) = p_{s,t}(G-K)$,  thus contradicting the assumption that $(G,s,t)$ is a counterexample.
\end{proof}

Let us give a quick corollary of Lemma~\ref{lemma-concave}.

\begin{lemma}\label{lemma-funnelbase-ineq}
For all integers $n\ge m\ge 2$, we have
$$b(m)b(n-m+2)\ge b(2)b(n).$$
\end{lemma}
\begin{proof}
Since the sequence $\frac{b(k+1)}{b(k)}$ is decreasing for $k\ge 2$ by Lemma~\ref{lemma-concave},
we have $\tfrac{b(l+1)}{b(l)}\le \tfrac{b(k+1)}{b(k)}$
for every $l\ge k\ge 2$.  Taking the product of these inequalities with $l$ replaced by $l, l+1, \ldots, l+t-1$
and $k$ replaced by $k, \ldots, k+t-1$ in turn, we have $\tfrac{b(l+t)}{b(l)}\le \tfrac{b(k+t)}{b(k)}$
for every $t\ge 0$.  The claim of the lemma follows by setting $k=2$, $l=m$ and $t=n-m$.
\end{proof}

Let us now derive a bound on $q_r(T)/a(r)$ where $T$ is a special kind of a tree with root $r$.

\begin{lemma}\label{lemma-funnelbase}
Let $T$ be a tree with root $r$ of degree $d\ge 1$ and with $l$ leaves distinct from the root.
Suppose that no vertex of $T$ is adjacent to more than one leaf of $T$
distinct from $r$.  Let $S$ be the set of non-leaf vertices of $T$ distinct from the root.
Then $$\prod_{v\in S} b(\deg(v))\ge \phi(l,d).$$
\end{lemma}
\begin{proof}
We prove the claim by induction on the number of vertices of $T$.
Without loss of generality, we can assume that all vertices in $S$ of degree two are adjacent to leaves distinct from $r$,
as otherwise they can be suppressed.
If all vertices in $S$ have degree two, then $l=d$ and $|S|\ge l-1$ by the assumption that $r$ is adjacent to at most one leaf,
and thus the inequality holds.

Otherwise, let $w$ be a vertex of $S$ of degree $m\ge 3$ that is farthest from $r$ and let $D$ denote the set of vertices of $S$ separated
from $r$ by $w$; all vertices of $D$ are adjacent to $w$ and have degree two.
If $|D|=m-1$, then a vertex of $D$ can be suppressed; hence, we can assume that $|D|=m-2$ and $w$ is incident with a leaf.

Note that the condition that no vertex is adjacent to more than one leaf of $T_Y$ distinct from $r$ follows from the fact that every leaf of $G[Y]-s$ is adjacent to $s$ (by Lemma~\ref{lemma-2conn}) and the fact that $G$ does not have parallel edges (and this is indeed the only place in the proof where we take advantage of this assumption). we have $l\ge d+m-2$, and thus $l-d+2\ge m>2$.  Hence,
$b(m)b(l-m-d+4)\ge b(2)b(l-d+2)$ by Lemma~\ref{lemma-funnelbase-ineq}, and thus
$$\frac{\phi(l,d)}{\phi(l-m+2,d)}=b(2)^{m-2}\frac{b(l-d+2)}{b(l-m-d+4)}\le b(m)b(2)^{m-3}.$$
Let $T'$ be the tree obtained from $T$ by deleting $D$ and the leaves adjacent to $D$ (so $w$ is a vertex
of degree two in $T'$).  By the induction hypothesis applied to $T'$, we have
$$b(2)\prod_{v\in S\setminus (D\cup \{w\})} b(\deg(v))\ge \phi(l-m+2,d).$$
Therefore,
\begin{align*}
\prod_{v\in S} b(\deg(v))&=b(m)b(2)^{m-2}\prod_{v\in S\setminus (D\cup\{w\})} b(\deg(v))\\
&\ge b(m)b(2)^{m-3}\phi(l-m+2,d)\ge \phi(l,d).
\end{align*}
\end{proof}

Let $G$ be a graph with distinct vertices $s$ and $t$.  For $r\in V(G)\setminus \{s\}$, a \emph{funnel with root $r$} in $G$
is a set of vertices $Y\subseteq V(G)$ such that $s,r\in Y$, $t\not\in Y\setminus\{r\}$, no vertex of $Y$ other than $s$ and $r$
has a neighbor outside of $Y$, and $G[Y]-s$ is a tree.  Clearly, a funnel contains at most $\deg_{G[Y]}(s)$ paths from $s$ to $r$.
Let $b(Y)=\prod_{v\in Y\setminus\{s,r\}} b(\deg(v))$.
We now apply Lemma~\ref{lemma-funnelbase} to the tree $T_Y$ obtained from $G[Y]$ by splitting $s$
into $\deg_{G[Y]}(s)$ leaves. Note that that condition that no vertex is adjacent to more than one leaf of $T_Y$ distinct from $r$
follows from the assumption that $G$ does not have parallel
edges (and this is indeed the only place in the proof where we take advantage of this assumption).

\begin{corollary}\label{cor-funnelbase}
Let $G$ be a graph with distinct vertices $s$ and $t$.
If $Y$ is a funnel with root $r$, then $b(Y)\ge \phi(\deg_{G[Y]}(s),\deg_{G[Y]}(r))$.
\end{corollary}

We are now ready to prove the main result.

\begin{proof}[Proof of Theorem~\ref{thm-paths}]
Suppose for a contradiction that $(G,s,t)$ is a smallest counterexample to Theorem~\ref{thm-paths}.
Let $Y_1$, \ldots, $Y_c$ be all distinct maximal funnels in $G$.  We additionally choose the counterexample
so that $|E(G[Y_1])|+\ldots+|E(G[Y_c])|$ is maximized among all smallest counterexamples.

Suppose for a contradiction that $Y_i \cap Y_j \neq \{s\}$ for some $1 \leq i < j \leq c$. Let $r_i$ and $r_j$ be the roots of $Y_i$ and $Y_j$, and choose $v \in (Y_i \cap Y_j) \setminus s$ so that the length of the unique path $P$ in $G[Y_i]$ with ends $v$ and $r_i$ in $G[Y_i]$ is minimum. Then $v \in \{r_i,r_j\}$, as otherwise
the neighbor of $v$ in $V(P)$ does not lie in $Y_j$, contradicting the definition of a funnel. Thus we may assume by symmetry that $r_j \in V_i$. It is now easy to verify that $Y_i \cup Y_j$ is a funnel with root $r_i$, in contradiction to our choice of  $Y_1, \ldots, Y_c$. 

It follows that any two of the subgraphs $G[Y_1]$, \ldots, $G[Y_c]$
intersect exactly in $s$. Moreover,  each edge of $G$ incident with $s$ belongs to exactly one of these subgraphs.  

Let $r_1$, \ldots, $r_c$ be the roots of the funnels; clearly the roots are pairwise distinct.  
For $i\in\{1,\ldots,c\}$, let $l_i=\deg_{G[Y_i]}(s)$, $d_i=\deg_{G[Y_i]}(r_i)$, and $n_i=\deg_G(r_i)-d_i$.
By symmetry, we can assume that $r_2$, \ldots, $r_c$ are distinct from $t$, and by the maximality of the funnels,
we conclude that $n_2,\ldots, n_c\ge 2$.

Let us first consider the case that $n_1\le 1$, and thus by the maximality of $Y_1$ we have $r_1=t$.
If $n_1=0$, then by Lemma~\ref{lemma-2conn} we have $c=1$.  Consequently,
$p_{s,t}(G)=l_1\le a(l_1)\phi(l_1,d_1)b(d_1)\le a(l_1)b(Y_1)b(d_1)=q_s(G)$
by Lemma~\ref{lemma-funnelweight} and Corollary~\ref{cor-funnelbase}. Hence, we can assume $n_1\ge 1$.

Suppose that $n_1=1$, and let $x$ be the neighbor of $t$ not in $Y_1$.  Note that $tx$ is a bridge in $G-s$, and thus
$p_{s,x}(G)=p_{s,t}(G)$.  Furthermore, $Y_1\cup\{x\}$, $Y_2$, \ldots, $Y_c$ are funnels in $G$ with specified vertices $s$ and $x$,
and by the choice of $(G,s,t)$, we conclude that $(G,s,x)$ is not a counterexample.  Hence, we have $p_{s,x}(G)\le q_s(G)$.
However, this contradicts the assumption that $(G,s,t)$ is a counterexample.

Therefore, we can assume $n_1\ge 2$.
Let $G'=G-((Y_1\cup \ldots\cup Y_c)\setminus \{r_1,\ldots,r_c\})$ and let $\beta=\prod_{v\in V(G')\setminus \{r_1,\ldots,r_c\}}b(\deg(v))$.
Since $G$ is a smallest counterexample and using Lemma~\ref{lemma-main-ineq} and Corollary~\ref{cor-funnelbase}, we have
\begin{align*}
p_{s,t}(G)&=\sum_{i=1}^c l_ip_{r_i,t}(G')\le \sum_{i=1}^cl_iq_{r_i}(G')\\
&=\beta\cdot\Bigl(\sum_{i=1}^c \frac{l_ia(n_i)}{b(n_i)}\Bigr)\cdot\prod_{i=1}^c b(n_i)
=\beta\cdot\Bigl(\sum_{i=1}^c l_i\alpha(n_i)\Bigr)\cdot\prod_{i=1}^c b(n_i)\\
&\le \beta \cdot a\Bigl(\sum_{i=1}^c l_i\Bigr)\prod_{i=1}^c \phi(l_i,d_i)b(n_i+d_i)\\
&=\beta \cdot a(\deg_G(s)) \prod_{i=1}^c \phi(l_i,d_i)b(\deg_G(r_i))\\
&\le \beta \cdot a(\deg_G(s)) \prod_{i=1}^c b(Y_i)b(\deg_G(r_i))=q_s(G).
\end{align*}
This contradicts the assumption that $G$ is a counterexample.  Hence, no counterexample to Theorem~\ref{thm-paths} exists.
\end{proof}

\section{Postponed proofs}\label{sec:post_proof}

We use the Stirling bound for factorials in the following form.
\begin{lemma}\label{lemma-stirling}
For every integer $k\ge 1$,
$$\sqrt{2\pi}\le \frac{k!e^k}{k^{k+1/2}}\le e.$$
\end{lemma}

This gives us the following bound for $\lambda$.
\begin{corollary}\label{cor-lambda}
For every integer $k\ge 2$,
$$1-\log(\sqrt{2\pi})\ge k(k-1)\lambda(k)-\bigl(k-\log\bigl(\rho e\sqrt{k}\bigr)\bigr)\ge 0$$
\end{corollary}

Let us now give the proofs postponed from the previous section.
Let us start with Lemma~\ref{lemma-concave}, whose statement we repeat for convenience.

\begin{lemcon}[Repeated]
\sconcave
\end{lemcon}
\begin{proof}
Equivalently, we need to show that the function $\lambda$ is positive and decreasing.
Since $\rho<2$, we have $k^k>\rho k!$, and thus $\lambda(k)>0$ for every $k\ge 2$.
Furthermore, using Corollary~\ref{cor-lambda}, we have
\begin{align*}
\lambda&(k)-\lambda(k+1) \\ &\ge \frac{k(k+1)-(k+1)\log\bigl(\rho e\sqrt{k}\bigr)-(k+1)(k-1)+(k-1)\log \bigl(\rho\sqrt{2\pi(k+1)}\bigr)}{(k-1)k(k+1)}\\
&\ge \frac{k+1-(k+1)\log\bigl(\rho e\sqrt{k}\bigr)+(k-1)\log \bigl(\rho\sqrt{2\pi k}\bigr)}{(k-1)k(k+1)}\\
&\ge \frac{(k-1)\log \sqrt{2\pi}-2\log\bigl(\rho\sqrt{k}\bigr)}{(k-1)k(k+1)},
\end{align*}
which is positive for $k\ge 2$; hence, $\lambda(k)$ is decreasing for $k\ge 2$.
\end{proof}

Since the function $\lambda$ is decreasing, the function $\alpha(k)$ is decreasing for $k\ge 3$.
Actually, $\alpha(k)$ is decreasing for $k\ge 2$, as can be verified by calculating the values $\alpha(2)>1.61$
and $\alpha(3)<1.37$.

We will need an additional monotonicity result.
\begin{lemma}\label{l:b}
The sequence $b(k)/k$ is decreasing for $k\ge 1$. 
\end{lemma}
\begin{proof}
We need to show that $$\frac{b(k+1)}{k+1} < \frac{b(k)}{k}$$ for every $k \geq 1$.
This is true for $k=1$, since $\rho<2$.  For $k\ge 2$, we equivalently need to verify that $\lambda(k)<\log (1+1/k)$.
Since $\log(1+x)\ge x - x^2/2$ for $x\ge 0$, by Corollary~\ref{cor-lambda},
it suffices to prove that
$$\frac{1}{k(k-1)}\Bigl(k-\log\bigl(\rho\sqrt{2\pi k}\bigr)\Bigr)<\frac{2k-1}{2k^2}.$$
This is equivalent to
$$k\log (2\rho^2\pi k)>3k-1.$$
This is true, since $\log (2\rho^2\pi k)>3$ for $k\ge 3$ and the case $k=2$ can be verified directly.
\end{proof}

We are now ready to prove Lemma~\ref{l:a}, whose statement we repeat for convenience.

\begin{lemsla}[Repeated]
\sla
\end{lemsla}
\begin{proof}
We have $a(0)=1$, and a direct calculation shows that $1.1 < a(1) < 1.2$, $1.7<a(2)<1.8$ and $a(3)>1.9$;
hence, $a(0)<a(1)<a(2)<a(3)$.  For $k \geq 4$, we have (using the definitions, (\ref{e:brecursion})
to replace $b^k(k+1)$ and $b^{k-1}(k)$, and Lemma~\ref{l:b})
\begin{align*}
\frac{a(k)}{a(k-1)}&=\frac{\alpha(k)b(k)}{\alpha(k-1)b(k-1)}=\frac{b(k)}{b(k-1)(1+k(\lambda(k-1)-\lambda(k)))}\\
&\ge \frac{b(k)}{b(k-1)e^{k(\lambda(k-1)-\lambda(k))}}=\frac{b^{k-1}(k-1)b^k(k+1)}{b^{2k-1}(k)}\\
&=\frac{kb^{k-1}(k-1)}{b(k)\cdot b^{k-1}(k)}=\frac{kb(k-1)}{(k-1)b(k)}>1.
\end{align*}
\end{proof}

\begin{corollary}\label{c:ka}
The sequence $k\alpha(k)$ is increasing for $k\ge 1$.
\end{corollary}
\begin{proof}
Since $k\alpha(k)=\frac{a(k)}{b(k)/k}$, the claim follows from Lemmas~\ref{l:b} and \ref{l:a}.
\end{proof}

Lemma~\ref{lemma-funnelweight} is straightforward to prove.

\begin{lemfun}[Repeated]
\sfunnelweight
\end{lemfun}
\begin{proof}
For $l\ge 1$, let us define $p(l)=a(l)\min\{\phi(l,d)b(d):1\le d\le l\}-l$;
hence, we aim to prove that $p(l)\ge 0$ for every $l\ge 1$.
The claim is trivial for $l=1$.  For $l=d=2$,
we have $a(2)\phi(2,2)b(2)=\alpha(2)\rho^3=2$ by the definition of $\alpha(2)$;
and for $l=2$ and $d=1$, we have $a(2)\phi(2,1)b(1)=a(2)b(3)>2.53$;
hence, $p(2)=0$.  We have numerically verified that $p(3)>0.3$, $p(4)>0.9$, $p(5)>1$, $p(6)>2$, $p(7)>3$,
$p(8)>5$, and $p(9)>7$.

We now prove by induction that $p(l)>7$ for $l\ge 9$.
We just established the basic case $l=9$.
Suppose now $l\ge 10$, and consider any integer $d\in\{1,\ldots,l\}$.  Let $d'=1$ if $d=1$ and $d'=d-1$ otherwise.
Lemmas~\ref{lemma-concave} and \ref{l:a} imply $$a(l)\phi(l,d)b(d)\ge a(l-1)\phi(l-1,d')b(d')\cdot \rho,$$
and thus $p(l)+l\ge (p(l-1)+l-1)\rho$ and
$$p(l)\ge p(l-1)+(\rho-1) (p(l-1) + l)-\rho>7+17(\rho-1)-\rho>7,$$
as required.
\end{proof}

We now start our work towards proving the crucial Lemma~\ref{lemma-main-ineq}.
For integers $k, n\ge 2$ and $l\ge 1$, let $$f(k,n,l)=k\alpha(k)\Bigl(\log \frac{\rho^{l-1}b(n+l)}{b(n)} - l \lambda(k)\Bigr).$$
We say that a triple $(k,n,l)$ is \emph{well-behaved} if $f(k,n,l) \geq l(\alpha(n)-\alpha(k))$.
The motivation for this definition is as follows.  Consider the situation in Lemma~\ref{lemma-main-ineq}
in the special case that $l_i=d_i$ for $i=1,\ldots,c$ (it is easy to reduce the proof to this case), and let
$k=\sum_{i=1}^c l_i$.  If all triples $(k,n_i,l_i)$ for $i=1,\ldots,c$ are well-behaved, the inequality of Lemma~\ref{lemma-main-ineq}
follows by a straightforward manipulation.
Let us remark that
\begin{equation}\label{eq:casel1}
f(k,n,1)=k\alpha(k)(\lambda(n)-\lambda(k)),
\end{equation}
and in particular $f(k,k,1)=0$ for every $k\ge 2$; hence, the triple $(k,k,1)$ is well-behaved.
The next lemma allows us to generate many additional well-behaved triples.

\begin{lemma}\label{l:induct1}
Let $(k,n,l)$ be a well-behaved triple of integers, where $k,n\ge 2$ and $l\ge 1$.
\begin{enumerate}[\textnormal{(\alph*)}]
\item If $k \geq \max (n,3)$, then the triple $(k+1,n,l)$ is well-behaved.
\item If $n \ge \max (k, 3)$ and $l=1$, then the triple $(k,n+1,1)$ is well-behaved.
\end{enumerate}
\end{lemma}
\begin{proof}
Note that the definition of $\alpha$ gives
\begin{equation}\label{e:alpha}
(x+1)\alpha(x+1)(\lambda(x+1)-\lambda(x)) = \alpha(x+1)-\alpha(x).
\end{equation}
for every integer $x\ge 3$.
\begin{enumerate}[(a)]
\item We have $f(k,n,l)\ge l(\alpha(n)-\alpha(k)) \geq 0$ as $(k,n,l)$ is well-behaved, $k\geq n\ge 2$ and $\alpha(x)$ is decreasing for $x\ge 2$.
Using (\ref{e:alpha}) and Corollary~\ref{c:ka} we get
\begin{align*}
f&(k+1,n,l)  \\
&=\frac{(k+1)\alpha(k+1)}{k\alpha(k)}f(k,n,l) - l(k+1)\alpha(k+1)(\lambda(k+1)-\lambda(k)) \\
&\geq f(k,n,l) - l(\alpha(k+1)-\alpha(k)) \geq  l(\alpha(n)-\alpha(k+1)),
\end{align*} as desired.
\item
By (\ref{eq:casel1}), Lemma~\ref{lemma-concave} and Corollary~\ref{c:ka}, and (\ref{e:alpha}), we get
\begin{align*}
f&(k,n+1,1)=f(k,n,1)-k\alpha(k)(\lambda(n)-\lambda(n+1)) \\
&\geq f(k,n,1)-(n+1)\alpha(n+1)(\lambda(n)-\lambda(n+1)) \\
&\geq (\alpha(n)-\alpha(k)) +\alpha(n+1)-\alpha(n) = \alpha(n+1) - \alpha(k).
\end{align*}
Consequently, the triple $(k,n+1,1)$ is well-behaved.
\end{enumerate}
\end{proof}	 

\begin{corollary}\label{c:well}
The triple $(k,n,1)$ is well-behaved for all pairs of positive integers $k$ and $n$ such that $k,n \geq 2$.
Moreover, the triple $(k,2,3)$ is well-behaved for all $k \geq 4$, and $(k,3,4)$ is well-behaved for all $k \geq 5$.
\end{corollary}
\begin{proof}
We have numerically verified that triples $(2,3,1)$, $(3,2,1)$, $(4,2,3)$, and $(5,3,4)$
are well-behaved.  As $(n,n,1)$ is well-behaved for all $n\ge 2$, and the triples $(2,3,1)$ and $(3,2,1)$
are well-behaved, Lemma~\ref{l:induct1} implies that $(k,n,1)$ is well-behaved
for $k,n \geq 2$.

Moreover, since $(4,2,3)$ is well-behaved, Lemma~\ref{l:induct1} implies that $(k,2,3)$ is well-behaved for all
$k \geq 4$.  Similarly, $(5,3,4)$ being well-behaved implies that $(k,3,4)$ is well-behaved for all $k \geq 5$.
\end{proof}

We need one more technical computational lemma.

\begin{lemma}\label{l:cases}
If integers $k \geq l \geq 2$, $n \geq 2$ satisfy
\begin{align}
\rho^{l-1}\frac{b(n+l)}{b(n)}&<\left(\frac{b(n+1)}{b(n)}\right)^l \label{e:t1}, \\
a(k)\rho^{l-1}b(n+l)-a(k-l)b(n)&<la(n) \label{e:t2}, \\
a(k)\rho^{l-1}b(n+l)-a(k-1)\rho^{l-2}b(n+l-1)&<a(n),  \label{e:t3}
\end{align}
then the triple $(k,n,l)$ is well-behaved.
\end{lemma}
\begin{proof}
We will show that if the conditions of the lemma are satisfied then $n=2$, $l=3$
and $k \geq 4$, or  $n=3$, $l=4$ and $k \geq 5$, implying the lemma by
Corollary~\ref{c:well}.

Since $b$ is increasing by Lemma~\ref{lemma-concave}, (\ref{e:t1}) implies $b(n+1)/b(n) > \rho$.
Since $b(n+1)/b(n)$ is decreasing by Lemma~\ref{lemma-concave} and $b(14)/b(13)<\rho$,
we have $$n \leq 12.$$ 
	
By Lemma~\ref{l:a}, $a$ is increasing, and thus (\ref{e:t2}) implies 
\begin{align}
a(l)(\rho^{l-1}b(n+l)- b(n))&\le a(k)(\rho^{l-1}b(n+l)- b(n))\nonumber\\
&\le a(k)\rho^{l-1}b(n+l)-a(k-l)b(n)<la(n)\label{e:t4}.
\end{align}	
We claim that $b(n+l) \geq b(n+2)\ge \alpha(n)b(n)=a(n)$ for all $n,l \geq 2$. Indeed, since $\alpha(5)<1$
and $\alpha$ is decreasing, we have $\alpha(n)<1$ and $b(n+2)>b(n)>\alpha(n)b(n)$ for $n\ge 5$, and we can verify the inequality for $n\in\{2,3,4\}$
by direct calculation.  Dividing the inequality (\ref{e:t4}) by $a(n)$ and using the fact that $\alpha$ is decreasing and $n\le 12$,
we obtain
$$l>a(l)\Bigl(\rho^{l-1}\frac{b(n+l)}{a(n)}-\frac{1}{\alpha(n)}\Bigr)\ge a(l)(\rho^{l-1}-1/\alpha(12)).$$
This only holds for $$l\le 39.$$
A direct computation for $2\le n\le 12$ and $2\le l\le 39$ shows that (\ref{e:t1}) and
(\ref{e:t4}) only hold when $(n,l)$ is contained in the set
$$P=\{(2,3),\ldots, (2,7), (3,4), \ldots, (3,9), (4,7), (4,8), (4,9)\}.$$
Let us remark that for $n=l=2$, we have $\rho\tfrac{b(4)}{b(2)}=\bigl(\tfrac{b(3)}{b(2)}\bigr)^2$ in contradiction
to (\ref{e:t1}) by the definition of $b$, while in all other cases not belonging to $P$, the inequalities fail by at least $0.03$.

If $k \geq l+10$, then since both $a$ and $b$ are increasing, by (\ref{e:t3}) we have
\begin{align*}
&a(l+10)\rho^{l-2}(\rho b(n+l)-b(n+l-1))\\
&\le a(k)\rho^{l-2}(\rho b(n+l)-b(n+l-1))\\
&\le a(k)\rho^{l-1}b(n+l)-a(k-1)\rho^{l-2}b(n+l-1)<a(n);
\end{align*}
a direct computation shows that among the possible pairs $(n,l)\in P$,
this is only satisfied by $(n,l)=(2,3)$ and $(n,l)=(3,4)$.
Hence, in this case we have $n=2$ and $l=3$ or $n=3$ and $l=4$, and $k\ge 13$.

To finish the proof, we verify by a direct computation that
the only triples $(k,n,l)$ satisfying $(n,l)\in P$, $l\le k\le l+9$, (\ref{e:t1}), (\ref{e:t2}), and (\ref{e:t3})
are those where $n=2$, $l=3$ and $k \geq 4$, or  $n=3$, $l=4$ and $k \geq 5$.
Let us remark that in the case $k=3$, $n=2$, $l=3$ we have
$a(3)\rho^2b(5)-a(0)b(2)=3a(2)$ by the definition of $\alpha(2)$ and $\alpha(3)$, contradicting (\ref{e:t2});
in all other cases, the inequalities fail by at least $0.005$.
\end{proof}

We are now ready to prove the variant of Lemma~\ref{lemma-main-ineq} where $l_i=d_i$ for $1\le i\le c$.

\begin{lemma}\label{l:main}
Let $c\ge 0$, $l_1, \ldots, l_c\ge 1$, and $n_1,\ldots,n_c\ge 2$ be integers.
Let $k = \sum_{i=1}^{c}l_i$.  Then
\begin{equation}\label{e:main2}
a(k)\prod_{i=1}^{c}\left(\rho^{l_i-1} \frac{b(n_i+l_i)}{b(n_i)} \right)\geq \sum_{i=1}^{c}l_i\alpha(n_i).
\end{equation}
\end{lemma}
\begin{proof}
We prove the claim by induction on $k$, and for a fixed $k$ by induction on $\sum_{i=1}^c (l_i-1)$.
The claim is trivial for $k=0$.
If $k=1$, then $c=l_1=1$, and we need to verify that $a(1)\tfrac{b(n+1)}{b(n)} \geq \alpha(n)$ for every $n \geq 2$.
This is trivial for $n\ge 5$, since $\alpha(5)<1$, $\alpha(n)$ is decreasing for $n\ge 2$, and $b$ is increasing.
For $n\in \{2,3,4\}$, the claim holds by the choice of $\alpha(1)$, as can be verified by a direct calculation.

Let us now assume $k\ge 2$.  Let $n=n_c$ and $l=l_c$, for brevity.  Suppose first that $l\ge 2$
and the triple $(k,n,l)$ does not satisfy at least one of the conditions (\ref{e:t1}), (\ref{e:t2}) and (\ref{e:t3})
\begin{itemize}
\item Let us consider the case that (\ref{e:t1}) is false, i.e.,
\begin{equation}\label{e:t1neg}
\rho^{l-1}\frac{b(n+l)}{b(n)}\ge \left(\frac{b(n+1)}{b(n)}\right)^l.
\end{equation}
Then we can replace $(n_c,l_c)$ by $l_c$ appearances of $(n_c,1)$,
keeping $k$ unchanged but decreasing the value of $\sum_{i=1}^c (l_i-1)$.  Applying the induction hypothesis
to the resulting sequence and combining the resulting inequality with (\ref{e:t1neg}), we obtain (\ref{e:main2})
as desired.

More precisely, by (\ref{e:t1neg}) and the induction hypothesis for the sequence $c+l-1$, $l_1,\ldots,l_{c-1},1,\ldots,1$, and $n_1,\ldots, n_{c-1}, n,\ldots, n$,
we have
\begin{align*}
a(k)\prod_{i=1}^{c}\left(\rho^{l_i-1} \frac{b(n_i+l_i)}{b(n_i)} \right)&\geq a(k)\left(\frac{b(n+1)}{b(n)}\right)^l\cdot \prod_{i=1}^{c-1}\left(\rho^{l_i-1} \frac{b(n_i+l_i)}{b(n_i)} \right)\\
&\geq l\alpha(n)+\sum_{i=1}^{c-1}l_i\alpha(n_i)=\sum_{i=1}^{c}l_i\alpha(n_i),
\end{align*}
as required.

\item Suppose now that (\ref{e:t2}) is false, i.e.,
\begin{equation}\label{e:t2neg}
a(k)\rho^{l-1}\frac{b(n+l)}{b(n)}\ge a(k-l)+ l\alpha(n). 
\end{equation}
In this case, we can remove $(n_c, l_c)$, decreasing $k$, and obtain the desired inequality from the induction hypothesis applied to the
resulting sequence, combined with (\ref{e:t2neg}).

More precisely, by (\ref{e:t2neg}) and the induction hypothesis for the sequence $c-1$, $l_1,\ldots,l_{c-1}$, and $n_1,\ldots, n_{c-1}$,
we have
\begin{align*}
a(k)\prod_{i=1}^{c}\left(\rho^{l_i-1} \frac{b(n_i+l_i)}{b(n_i)} \right)&\ge (a(k-l)+ l\alpha(n))\prod_{i=1}^{c-1}\left(\rho^{l_i-1} \frac{b(n_i+l_i)}{b(n_i)} \right)\\
&\ge l\alpha(n)+a(k-l)\prod_{i=1}^{c-1}\left(\rho^{l_i-1} \frac{b(n_i+l_i)}{b(n_i)} \right)\\
&\ge l\alpha(n)+\sum_{i=1}^{c-1}l_i\alpha(n_i)=\sum_{i=1}^c l_i\alpha(n_i).
\end{align*}

\item Finally, suppose that (\ref{e:t3}) is false, i.e.,
\begin{equation}\label{e:t3neg}
a(k)\rho^{l-1}\frac{b(n+l)}{b(n)}\ge a(k-1)\rho^{l-2}\frac{b(n+l-1)}{b(n)}+ \alpha(n).
\end{equation}
In this case, we can replace $(n_c, l_c)$ by $(n_c, l_c-1)$, decreasing $k$, and obtain the desired inequality from the induction hypothesis applied to the
resulting sequence, combined with (\ref{e:t3neg}).

More precisely, by (\ref{e:t3neg}) and the induction hypothesis for the sequence $c$, $l_1,\ldots,l_{c-1},l_c-1$, and $n_1,\ldots, n_c$,
we have
\begin{align*}
a(k)&\prod_{i=1}^{c}\left(\rho^{l_i-1} \frac{b(n_i+l_i)}{b(n_i)} \right) \\&\ge \Bigl(\alpha(n)+a(k-1)\rho^{l-2}\frac{b(n+l-1)}{b(n)}\Bigl)\prod_{i=1}^{c-1}\left(\rho^{l_i-1} \frac{b(n_i+l_i)}{b(n_i)} \right)\\
&\ge \alpha(n)+a(k-1)\rho^{l-2}\frac{b(n+l-1)}{b(n)}\prod_{i=1}^{c-1}\left(\rho^{l_i-1} \frac{b(n_i+l_i)}{b(n_i)} \right)\\
&\ge \alpha(n)+(l-1)\alpha(n)+\sum_{i=1}^{c-1}l_i\alpha(n_i)=\sum_{i=1}^c l_i\alpha(n_i).
\end{align*}
\end{itemize}

Therefore, we can assume that either $l_c=1$, or $k\ge l_c\ge 2$ and the triple $(k,n_c,l_c)$ satisfies the conditions (\ref{e:t1}), (\ref{e:t2}) and (\ref{e:t3}).
By Corollary~\ref{c:well} and Lemma~\ref{l:cases}, it follows that $(k,n_c,l_c)$ is well-behaved.
By symmetry, we can assume that the triple $(k,n_i,l_i)$ is well-behaved for $i=1,\ldots,c$.
Then 
 \begin{align*}  
  a(k)&\prod_{i=1}^{c}\left(\rho^{l_i-1} \frac{b(n_i+l_i)}{b(n_i)}\right) 
  \\& \stackrel{(\ref{e:brecursion})}{=} \alpha(k) \cdot k \left( \frac{b(k)}{b(k+1)}\right)^k \cdot \prod_{i=1}^{c}\left(\rho^{l_i-1} \frac{b(n_i+l_i)}{b(n_i)}  \right)
  	\\ &= k\alpha(k)\prod_{i=1}^{c}\left(\left( \frac{b(k)}{b(k+1)}\right)^{l_i} \rho^{l_i-1} \frac{b(n_i+l_i)}{b(n_i)} \right) 
 	 \\&= k\alpha(k) \prod_{i=1}^{c}e^{\frac{1}{k\alpha(k)}f(k,n_i,l_i)}=k\alpha(k) e^{\frac{1}{k\alpha(k)}\sum_{i=1}^cf(k,n_i,l_i)} 
 \\	& \geq k\alpha(k) \left(1 + \frac{1}{k\alpha(k)}\sum_{i=1}^cf(k,n_i,l_i)\right)
 	   \geq k\alpha(k) + \sum_{i=1}^cl_i(\alpha(n_i)-\alpha(k)) \\&=  \sum_{i=1}^cl_i \alpha(n_i),
 \end{align*}
as desired.
\end{proof}

It is now easy to finish the proof of Lemma~\ref{lemma-main-ineq}.

\begin{lemmain}[Repeated]
\smain
\end{lemmain}

\begin{proof}
Let $k=\sum_{i=1}^c l_i$.
Note that by Lemma~\ref{lemma-funnelbase-ineq}, we have
$$\phi(l_i,d_i)b(n_i+d_i)=\rho^{l_i-2}b(l_i-d_i+2)b(n_i+d_i)\ge \rho^{l_i-2}b(2)b(n_i+l_i)=\rho^{l_i-1}b(n_i+l_i)$$ for $i=1,\ldots,c$.
Using Lemma~\ref{l:main}, we obtain
\begin{align*}
a(k)\prod_{i=1}^c \phi(l_i,d_i)b(n_i+d_i)&\ge a(k)\prod_{i=1}^{c}\rho^{l_i-1}b(n_i+l_i)\\
&= a(k)\prod_{i=1}^{c}\left(\rho^{l_i-1} \frac{b(n_i+l_i)}{b(n_i)} \right) \cdot \prod_{i=1}^c b(n_i)\\
&\ge \Bigl(\sum_{i=1}^c l_i\alpha(n_i)\Bigr)\cdot\prod_{i=1}^c b(n_i),
\end{align*}
as desired.
\end{proof}

\noindent {\bf Acknowledgement.} This research was partially completed at the \emph{2019 Barbados Graph Theory Workshop} held at the Bellairs Research Institute. We thank the participants of the workshop for providing a stimulating enviroment for research, and Michelle Delcourt and Alex Roberts, in particular, for helpful discussions. Jon Noel would like to thank Jaehoon Kim for telling him about Kir\'{a}ly's conjecture and letting him share it with participants of the workshop.

\bibliographystyle{acm}
\bibliography{numcycles.bib}

\end{document}